\documentclass[12pt,reqno]{amsart}

\usepackage{cite}
\usepackage[utf8]{inputenc}
\usepackage{amsmath}
\usepackage{amsthm}
\usepackage{amssymb}
\usepackage{amsfonts}
\usepackage{latexsym}
\usepackage{hyperref}
\usepackage{cite}
\usepackage{xcolor}
\usepackage{fullpage}
\usepackage{longtable}

\usepackage{comment}
\usepackage{multirow}

\theoremstyle{plain}
\newtheorem{theorem}{Theorem}[section]
\newtheorem{corollary}[theorem]{Corollary}
\newtheorem{lemma}[theorem]{Lemma}

\theoremstyle{definition}

\newtheorem*{remark}{Remark}

\usepackage{algpseudocode}

\newcommand{\li}{\operatorname{li}}

\newcommand{\logb}[1]{\log\!\left(#1\right)}

\renewcommand{\Re}{\operatorname{Re}}

\begin{document}

\title{Sharper bounds for the error in the prime number theorem assuming the Riemann Hypothesis}
	
\author[E.~S.~Lee]{Ethan~Simpson~Lee}
\address{University of the West of England, School of Computing and Creative Technologies, Coldharbour Lane, Bristol, BS16 1QY} 
\email{ethan.lee@uwe.ac.uk}
\urladdr{\url{https://sites.google.com/view/ethansleemath/home}}

\author[P.~Nosal]{Pawe{\l}~Nosal}
\address{University of Warwick, Mathematics Institute, Zeeman Building, Coventry, CV4 7AL}
\email{pawel.nosal@warwick.ac.uk}
\urladdr{\url{https://sites.google.com/view/pawelnosalmaths}}

\maketitle

\begin{abstract}
In this paper, we establish new bounds for classical prime-counting functions. All of our bounds are explicit and assume the Riemann Hypothesis. First, we prove that $|\psi(x) - x|$ and $|\vartheta(x) - x|$ are bounded from above by
\begin{equation*}
    \frac{\sqrt{x}\log{x}(\log{x} - \log\log{x})}{8\pi} 
\end{equation*}
for all $x\geq 101$ and $x \geq 2\,657$ respectively, where $\psi(x)$ and $\vartheta(x)$ are the Chebyshev $\psi$ and $\vartheta$ functions. Using the extra precision offered by these results, we also prove new explicit descriptions for the error in each of Mertens' theorems which improve earlier bounds by Schoenfeld. 
\end{abstract}

\section{Introduction}

Between 1962 and 1976, Rosser and Schoenfeld published a series of foundational papers \cite{RosserSchoenfeld1962, RosserSchoenfeld, Schoenfeld} which describe the error in a collection of approximations for widely applicable functions over primes. The most important functions they studied are 
\begin{equation*}
    \psi(x) := \sum_{n\leq x} \Lambda(n), \quad 
    \vartheta(x) = \sum_{p\leq x} \log{p} , \quad
    \text{and}\quad
    \pi(x) = \sum_{p\leq x} 1 ,
\end{equation*}
in which $\Lambda(n)$ is the von Mangoldt function and $p$ are prime numbers. Recall that the prime number theorem tells us $\psi(x) \sim x$, $\vartheta(x)\sim x$, and $\pi(x)\sim \li(x)$, where $\li(x)$ is the logarithmic integral; these statements are equivalent. We generally use analytic techniques to study the error in $\psi(x) \sim x$, then use this knowledge to study the error in $\vartheta(x)\sim x$ and $\pi(x)\sim \li(x)$. 
Using their results on the error in $\psi(x) \sim x$, Rosser and Schoenfeld also studied the error in Mertens' theorems, which state
\begin{equation*}
    \sum_{p\leq x} \frac{\log{p}}{p} \sim \log{x} + E, \quad
    \sum_{p\leq x} \frac{1}{p} \sim \log\log{x} + B, \quad\text{and}\quad
    \prod_{p\leq x} \left(1 - \frac{1}{p}\right) \sim \frac{e^{-C}}{\log{x}} ,
\end{equation*}
in which $B = 0.26149\dots$, $C = 0.57721\dots$, and $E = - 1.33258\dots$; these statements are also equivalent. In this paper, we prove new conditional bounds for the error in each of the aforementioned approximations, which are cornerstone results in analytic number theory. 

\begin{remark}
Bounds for the functions in Mertens' theorems and $\vartheta(x)$ are often more practical to work with or apply than bounds for $\psi(x)$. For example, the product in Mertens' theorems appears naturally in sieve methods, because it approximates the density of integers that are not divisible by small primes.
\end{remark}


{To begin,} there is an explicit connection between the error in the approximation $\psi(x) \sim x$ and the distribution of the zeros of the Riemann zeta-function $\zeta(s)$, where $s$ always denotes a complex number. Recall that every zero of $\zeta(s)$ is either a \textit{trivial} or a \textit{non-trivial} zero. The trivial zeros of $\zeta(s)$ occur at $s=-2n$ for every positive integer $n$ and the non-trivial zeros of $\zeta(s)$ are complex and lie inside the critical strip $0 \leq \Re{s} \leq 1$. 

Suppose that $\varrho = \beta + i\gamma$ denotes a non-trivial zero of $\zeta(s)$. The prime number theorem famously follows from the observation $\beta \neq 1$ and the Riemann Hypothesis (RH) postulates that every non-trivial zero $\varrho$ of $\zeta(s)$ satisfies $\beta = 1/2$. The RH has been verified for a significant number of zeros (see \cite{PlattTrudgianZeros}) and we expect to obtain the strongest error term in the prime number theorem by assuming the RH. In fact, if the RH is assumed to be true, then we can collect observations from Schoenfeld \cite{Schoenfeld} and Greni\'{e}--Molteni \cite{Grenie2019} to see that 
\begin{equation}\label{eqn:Schoenfeld}
    |\psi(x) - x| \leq 
    \begin{cases}
        \left(\frac{\log{x}}{8\pi} + 2\right) \sqrt{x}\log{x} &\text{for all $x\geq 2$,}\\
        \frac{\sqrt{x}(\log{x})^2}{8\pi} &\text{for all $x\geq 73.2$,} \\
        \frac{\sqrt{x}\log{x} (\log{x} - 2)}{8\pi} &\text{for all $x\geq 2.3\cdot 10^9$.}
    \end{cases}
\end{equation}

The bounds in \eqref{eqn:Schoenfeld} clearly suggest that as the input $x$ increases, we are able to gain better quantitative control over the error in the prime number theorem. We prove the following results (Theorems \ref{thm:GoldstonExplicit} and \ref{thm:hauptsatz}), which formalise this expectation and improve the bounds in \eqref{eqn:Schoenfeld}. Theorem \ref{thm:hauptsatz}, particularly \eqref{eqn:pnt_psi}, is stronger than Theorem \ref{thm:GoldstonExplicit} on $x \leq e^{30\,369.581\dots}$ and includes an error bound of the same shape for $\vartheta(x) \sim x$. 

\begin{theorem}\label{thm:GoldstonExplicit}
If the RH is true and $x\geq 11$, then
\begin{equation*}
    \left|\psi(x) - x \right| \leq \sqrt{x}\log{x} \left(\frac{\log{x}}{8\pi} - \left(\frac{1}{2\pi} + \frac{1.465}{\log{x}}\right) \log\log{x} + 1.2325\right) . 
\end{equation*}
\end{theorem}

\begin{theorem}\label{thm:hauptsatz}
If the RH is true, then
\begin{align}
    \left|\psi(x) - x \right|
    &\leq \frac{\sqrt{x}\log{x}(\log{x} - \log\log{x})}{8\pi} \quad\text{for all}\quad x \geq 101, \label{eqn:pnt_psi} \\
    \left|\vartheta(x) - x \right|
    &\leq \frac{\sqrt{x}\log{x}(\log{x} - \log\log{x})}{8\pi} \quad\text{for all}\quad x \geq 2\,657 . \label{eqn:pnt_theta} 
\end{align}
\end{theorem}


{The extra precision offered by Theorems \ref{thm:GoldstonExplicit} and \ref{thm:hauptsatz} also lead to improved explicit knowledge of Mertens' theorems.} In particular, we prove the following corollary of these results, which gives sharp ranges of $x$ (accurate to one decimal place) in which lower-order terms in the error of each of Mertens' theorems can be omitted. {These results improve Schoenfeld's related bounds, presented in \cite[Cor.~2-3]{Schoenfeld}, and provide user-friendly bounds with practical utility for further analytical or computational applications.}

\begin{corollary}\label{cor:broader}
If the RH is true, then
\begin{align}
    \left|\sum_{p\leq x} \frac{\log{p}}{p} - \log{x} - E \right|
    &\leq \frac{3(\log{x})^2}{8\pi\sqrt{x}} \quad\text{for all}\quad x \geq 43.1, \label{eqn:moi_1}\\
    \left|\sum_{p\leq x} \frac{1}{p} - \log\log{x} - B \right|
    &\leq \frac{3\log{x}}{8\pi\sqrt{x}} \quad\text{for all}\quad x \geq 24.4 , \label{eqn:moi_2}\\
    \left|e^C \log{x} \prod_{p\leq x} \left(1 - \frac{1}{p}\right) - 1 \right|
    &\leq \frac{3\log{x}}{8\pi\sqrt{x}} \quad\text{for all}\quad x \geq 23.8 , \label{eqn:moi_3}\\
    \left|\frac{e^{-C}}{\log{x}} \prod_{p\leq x} \left(1 - \frac{1}{p}\right)^{-1} - 1 \right|
    &\leq \frac{3\log{x}}{8\pi\sqrt{x}} \quad\text{for all}\quad x \geq 24.2 . \label{eqn:moi_4}
\end{align}
\end{corollary}


\subsection*{Structure}

The remainder of this document is organised as follows. In Section \ref{sec:technical_lemmas}, we introduce some auxiliary bounds which will be required to prove Theorems \ref{thm:GoldstonExplicit} and \ref{thm:hauptsatz}. In Section \ref{sec:proofs}, we prove Theorem \ref{thm:GoldstonExplicit}, Theorem \ref{thm:hauptsatz}, and Corollary \ref{cor:broader} using these auxiliary bounds. 

\subsection*{Methodology}

For $x > 10^{19}$, we prove Theorem \ref{thm:GoldstonExplicit} by making a method of Goldston \cite{Goldston} explicit, and Theorem \ref{thm:hauptsatz} by utilising a smoothed explicit formula for $\psi(x)$ from Rosser and Schoenfeld \cite{RosserSchoenfeld}. Computations by B\"{u}the \cite{Buthe} are used to verify Theorems \ref{thm:GoldstonExplicit} and \ref{thm:hauptsatz} for each $x \leq 10^{19}$. Once Theorems \ref{thm:GoldstonExplicit} and \ref{thm:hauptsatz} are proved, we apply them in standard arguments to establish each result in Corollary \ref{cor:broader}. 

\subsection*{Acknowledgements}

ESL thanks the Heilbronn Institute for Mathematical Research for their support. PN thanks the Bristol School of Mathematics Summer Research Bursary for their support. We also thank Adrian Dudek and Daniel Johnston for bringing the papers \cite{Buthe, Goldston} to our attention, as well as other colleagues for their support, discussions, and valuable feedback. 

\section{Auxiliary Bounds on the Distribution of Zeros}\label{sec:technical_lemmas}

In this section, we introduce several important technical results that will be referred to throughout this paper. We always write $\varrho = \beta + i\gamma$ to denote a non-trivial (i.e., complex) zero of $\zeta(s)$.

If $T\geq 2\pi$, then the number $N(T)$ of $\varrho$ such that $0<\gamma \leq T$ satisfies the relationship
\begin{align}
    \left|N(T) - \frac{T}{2\pi} \logb{\frac{T}{2\pi e}} + \frac{7}{8}\right| 
    &\leq \min\{0.28\log{T}, 0.1038\log{T} + 0.2573\log\log{T} + 9.3675\} \nonumber\\
    &:= R(T) . \label{eqn:NT_est}
\end{align}
This result is a combination of observations from Brent--Platt--Trudgian \cite[Cor.~1]{BPT_mean_square} and Hasanalizade--Shen--Wong \cite[Cor.~1.2]{HasanalizadeShenWong}. A straightforward consequence of \eqref{eqn:NT_est} is that
\begin{equation}\label{eqn:use_me_instead}
    N(T) 
    \leq \frac{T\log{T}}{2\pi} .
\end{equation}
Next, exact computations for certain sums over zeros have been done using extensive databases of zeros. For example, the computations in \cite[Tab.~1]{FioriKadiriSwidinsky} tell us that
\begin{equation}\label{eqn:hafaliad3}
    \sum_{0<|\gamma| \leq 10^7} \frac{1}{|\gamma|} \leq \omega_1 := 16.2106480369 .
\end{equation}
When exact computations are not available or applicable, we require theory to bound these sums. To this end, we import an auxiliary bound from Skewes \cite[Lem.~1(ii)]{skewes}, namely
\begin{equation}\label{eqn:Skewes}
    \sum_{|\gamma| \geq T} \frac{1}{\gamma^2} < \frac{\log{T}}{\pi T} 
    \quad\text{for all}\quad T\geq 1 ,
\end{equation}
and the following result from Lehman \cite[Lem.~1]{Lehman}.

\begin{lemma}[Lehman]\label{lem:Lehman}
If $\phi(t)$ is a continuous, positive, non-decreasing function on $2\pi e \leq U \leq t \leq V$, then
\begin{equation*}
    \sum_{\substack{U < \gamma \leq V}} \phi(\gamma)
    \leq \frac{1}{2\pi} \int_{U}^{V} \phi(t)\logb{\frac{t}{2\pi}}\,dt + 4 \phi(U)\log{U} + 2\int_{U}^{\infty} \frac{\phi(u)}{u}\,du .
\end{equation*}
\end{lemma}

\begin{remark}
Brent, Platt, and Trudgian have refined Lemma \ref{lem:Lehman} in \cite[Cor.~1]{BrentAccurate}. However, we do not require their improved result, which is slightly more technical to apply, for our purposes. 
\end{remark}

\section{Proof of Main Results}\label{sec:proofs}

In this section, we prove each of our main results (Theorem \ref{thm:GoldstonExplicit}, Theorem \ref{thm:hauptsatz}, and Corollary \ref{cor:broader}). 

\subsection{Proof of Theorem \ref{thm:GoldstonExplicit}}\label{ssec:thm:GoldstonExplicit}

To prove Theorem \ref{thm:GoldstonExplicit} for $x \geq 10^{19}$, we make explicit a method of Goldston \cite{Goldston}. To this end, note that $\psi(x)$ is non-decreasing, so
\begin{equation}\label{eqn:hafaliad1}
    \frac{\psi_1(x-h) - \psi_1(x)}{-h} 
    \leq \psi(x) \leq
    \frac{\psi_1(x+h) - \psi_1(x)}{h} ,
\end{equation}
for some $1\leq h \leq x/2$. Next, note that
\begin{equation}\label{eqn:explicit_fmla_psi1}
    \psi_1(x) 
    = \int_2^x \psi(t)\,dt
    = \sum_{n\leq x} \Lambda(n) (x-n)
    = \frac{x^2}{2} - \sum_{\varrho} \frac{x^{\varrho+1}}{\varrho (\varrho +1)} - x \log(2\pi) + \epsilon(x) ,
\end{equation}
in which $x\not\in\mathbb{Z}$ and $1.545 < \epsilon(x) < 2.069$; see \cite[Lem.~3]{CullyHugillDudek}. The explicit formula for $\psi_1(x)$ in \eqref{eqn:explicit_fmla_psi1} holds for $x\not\in\mathbb{Z}$, but one can circumvent any issues when $x\in\mathbb{Z}$ by writing
\begin{equation*}
    \psi_1(x) = \frac{\psi_1(x+0^+) + \psi_1(x-0^+)}{2} .
\end{equation*}
The explicit formula \eqref{eqn:explicit_fmla_psi1} implies
\begin{equation}\label{eqn:hafaliad2}
    \Bigg| \frac{\psi_1(x\pm h) - \psi_1(x)}{\pm h}
    - x \mp \frac{h}{2} \pm \sum_{\varrho} \frac{(x\pm h)^{\varrho + 1} - x^{\varrho + 1}}{h \varrho (\varrho + 1)} + \logb{2\pi} \Bigg| \leq 1 .
\end{equation}
In the following lemmas, we approximate the sum over zeros in \eqref{eqn:hafaliad2} against a truncated sum over zeros, and provide a bound for the truncated sum. 

\begin{lemma}\label{lem:canlyniad1}
If the RH is true, $h = x/y$, and $y\geq H_1 := 10^7$, then  
\begin{equation*}
\begin{split}
    &\Bigg|\sum_{\varrho} \frac{(x\pm h)^{\varrho + 1} - x^{\varrho + 1}}{\pm  h \varrho (\varrho + 1)}
    - \sum_{|\gamma| \leq y} \frac{x^{\varrho}}{\varrho} \Bigg| 
    < \left(2.6 + 2 \left(1 + \frac{1}{y}\right)^{\frac{3}{2}} \right) \frac{\sqrt{x} \log{y}}{\pi}
    + \frac{2.6 \sqrt{x}}{\pi y} \left(\logb{\frac{y}{2\pi}}\right)^2 .
\end{split}
\end{equation*}
\end{lemma}

\begin{proof}
To begin, note that
\begin{equation*}
    \sum_{|\gamma| \leq y} \frac{(x\pm h)^{\varrho + 1} - x^{\varrho + 1}}{\pm  h \varrho (\varrho + 1)}
    = \sum_{|\gamma| \leq y} \frac{x^{\varrho}}{\varrho} + \sum_{|\gamma| \leq y} w_{\varrho},
\end{equation*}
where
\begin{equation*}
    w_{\varrho} = x^{\varrho + 1} \left(\frac{(1\pm h/x)^{\varrho + 1} - 1 \mp (\varrho + 1) h/x}{h \varrho (\varrho + 1)}\right) .
\end{equation*}
It follows from the lemma in \cite[Sec.~2]{Goldston} that
\begin{equation*}
    |w_{\varrho}| 
    \leq 2.6 x^{\frac{3}{2}} \left(\frac{|\varrho + 1| (|\varrho + 1| + 1) (h/x)^2}{h |\varrho| |\varrho + 1|}\right) 
    \leq \frac{2.6 h}{\sqrt{x}} \left(\frac{|\varrho + 1| + 1}{|\varrho|}\right) 
    \leq \frac{2.6 h}{\sqrt{x}} \left(1 + \frac{2}{|\varrho|}\right) .
\end{equation*}
Therefore,
\begin{equation}\label{eqn:hafaliad4}
\begin{split}
    \Bigg|\sum_{\varrho} \frac{(x\pm h)^{\varrho + 1} - x^{\varrho + 1}}{\pm  h \varrho (\varrho + 1)}
    - \sum_{|\gamma| \leq y} \frac{x^{\varrho}}{\varrho} \Bigg| 
    &\leq \frac{2.6 h}{\sqrt{x}} \left(\sum_{|\gamma| \leq y} 1 + \sum_{|\gamma| \leq y} \frac{2}{|\gamma|}\right) + \frac{2 (x+h)^{\frac{3}{2}}}{h} \sum_{|\gamma| > y} \frac{1}{\gamma^2}.
\end{split}
\end{equation}
It follows from \eqref{eqn:use_me_instead} that
\begin{equation*}
    \frac{2.6 h}{\sqrt{x}} \sum_{|\gamma| \leq y} 1
    \leq \frac{2.6 h y \log{y}}{\pi \sqrt{x}}
    \leq \frac{2.6 \sqrt{x} \log{y}}{\pi} .
\end{equation*}
Next, \eqref{eqn:Skewes} implies
\begin{equation*}
    \frac{2 (x+h)^{\frac{3}{2}}}{h} \sum_{|\gamma| > y} \frac{1}{\gamma^2} 
    \leq \frac{2 (x+h)^{\frac{3}{2}} \log{y}}{\pi h y} 
    = \frac{2 (x+h)^{\frac{3}{2}} \log{y}}{\pi x} 
    = \frac{2}{\pi} \left(1 + \frac{1}{y}\right)^{\frac{3}{2}} \sqrt{x} \log{y} .
\end{equation*}
Finally, $y\geq H_1$, \eqref{eqn:hafaliad3}, and Lemma \ref{lem:Lehman} imply
\begin{align}
    \frac{2.6 h}{\sqrt{x}} \sum_{|\gamma| \leq y} \frac{2}{|\gamma |}
    &\leq \frac{5.2 h}{\sqrt{x}} \Bigg( \sum_{|\gamma| \leq H_1} + \sum_{H_1 < |\gamma| \leq y} \Bigg) \frac{1}{|\gamma |} \nonumber\\
    &\leq \frac{5.2 h}{\sqrt{x}} \left(\omega_1 + \frac{1}{\pi} \int_{H_1}^{y} \frac{1}{t}\logb{\frac{t}{2\pi}}\,dt + \frac{8\log{H_1} + 4}{H_1}\right) \nonumber\\
    &= \frac{5.2 h}{\sqrt{x}} \left(\omega_1 + \frac{1}{2\pi} \left(\left(\logb{\frac{y}{2\pi}}\right)^2 - \left(\logb{\frac{H_1}{2\pi}}\right)^2\right) + \frac{8\log{H_1} + 4}{H_1}\right) \nonumber\\
    &< \frac{2.6 \sqrt{x}}{\pi y} \left(\logb{\frac{y}{2\pi}}\right)^2 . \label{eqn:hafaliad5}
\end{align}
Insert these observations into \eqref{eqn:hafaliad4} to reveal the result. 
\end{proof}

\begin{lemma}\label{lem:canlyniad2}
If the RH is true, $y = \sqrt{x}/\log{x}$, and $x\geq 10^{19}$, then  
\begin{equation*}
    \Bigg|\sum_{|\gamma| \leq y} \frac{x^{\varrho}}{\varrho} \Bigg| 
    < \frac{\sqrt{x}\log{x}}{2\pi} \left(\frac{\log{x}}{4} - \log\log{x}\right) .
\end{equation*}
\end{lemma}

\begin{proof}
Following similar arguments to the deduction of \eqref{eqn:hafaliad5}, we see that 
\begin{equation*}
    \Bigg|\sum_{|\gamma| \leq y} \frac{x^{\varrho}}{\varrho} \Bigg| < \frac{\sqrt{x}}{2\pi}\left(\logb{\frac{y}{2\pi}}\right)^2 .
\end{equation*}
Now, 
\begin{align*}
    \left(\logb{\frac{\sqrt{x}}{2\pi\log{x}}}\right)^2
    &= \log{x} \left(\frac{\log{x}}{4} - \left(1 + \frac{\logb{2\pi}}{\log\log{x}} - \frac{(\logb{2\pi\log{x}})^2}{\log{x} \log\log{x}}\right) \log\log{x}\right) \\
    &< \log{x} \left(\frac{\log{x}}{4} - \log\log{x}\right) ,
\end{align*}
because the coefficient of $\log\log{x}$ increases in the range $e < x \leq 2.00299\cdot 10^{38}$, decreases in the range $x > 2.00299\cdot 10^{38}$, and tends to the limit $1$ as $x\to\infty$. The result follows naturally.
\end{proof}

Now, we combine these ingredients to prove the result.

\begin{proof}[Proof of Theorem \ref{thm:GoldstonExplicit}]
Apply \eqref{eqn:hafaliad2} and Lemma \ref{lem:canlyniad1} in \eqref{eqn:hafaliad1} with $h = x/y$ to see that if $y\geq H_1$ and the RH is true, then
\begin{align}
    \Bigg|\psi(x) - x + \sum_{|\gamma| \leq y} \frac{x^{\varrho}}{\varrho}\Bigg| &< \frac{x}{2y} + \left(2.6 + 2 \left(1 + \frac{1}{y}\right)^{\frac{3}{2}} \right) \frac{\sqrt{x} \log{y}}{\pi}
    + \frac{2.6 \sqrt{x}}{\pi y} \left(\logb{\frac{y}{2\pi}}\right)^2 + 2.84 \nonumber\\
    &\leq \frac{x}{2y} + 1.465 \sqrt{x} \log{y} . \label{eqn:hafaliad6}
\end{align}
Next, assert $y = \sqrt{x}/\log{x}$ in \eqref{eqn:hafaliad6}, which satisfies $ y\geq H_1$ on $x\geq 10^{19}$, to reveal
\begin{align*}
    \Bigg|\psi(x) - x + \sum_{|\gamma| \leq y} \frac{x^{\varrho}}{\varrho} \Bigg| 
    &\leq \frac{\sqrt{x}\log{x}}{2} + 1.465 \sqrt{x}\left(\frac{\log{x}}{2} - \log\log{x} \right) \\
    &\leq \left( \frac{2.465}{2} - \frac{1.465\log\log{x}}{\log{x}} \right) \sqrt{x}\log{x}
\end{align*}
Apply Lemma \ref{lem:canlyniad2} to see
\begin{align*}
    |\psi(x) - x| 
    &< \sqrt{x}\log{x} \left(\frac{\log{x}}{8\pi} - \left(\frac{1}{2\pi} + \frac{1.465}{\log{x}}\right) \log\log{x} + 1.2325 \right) \\
    &= \sqrt{x}\log{x} \left(\frac{\log{x}}{8\pi} - \left(\frac{1}{2\pi} + \frac{1.465}{\log{x}} - \frac{1.2325}{\log\log{x}}\right) \log\log{x} \right) .
\end{align*}
With this, the result has been established for all $x\geq 10^{19}$, so all that remains is to prove the result for $x < 10^{19}$. To this end, recall that B\"{u}the \cite{Buthe} has verified that
\begin{equation}\label{eqn:Buthe}
    |\psi(x) - x| \leq 0.94 \sqrt{x}
    \quad\text{for all}\quad
    11 \leq x \leq 10^{19} .
\end{equation}
If $x\geq 11$, then 
\begin{equation*}
    0.94\leq \log{x} \left(\frac{\log{x}}{8\pi} - \left(\frac{1}{2\pi} + \frac{1.465}{\log{x}} - \frac{1.2325}{\log\log{x}}\right) \log\log{x} \right) ,
\end{equation*}
so the result is proved for all $x\geq 11$.
\end{proof}

\subsection{Proof of Theorem \ref{thm:hauptsatz}}\label{ssec:hauptsatz}

Here, we prove \eqref{eqn:pnt_psi} and \eqref{eqn:pnt_theta} using Theorem \ref{thm:GoldstonExplicit}, computations from B\"{u}the \cite{Buthe}, and the following technical lemma (Lemma \ref{lem:RosserScheonfeldSmoothEFConsequence}). A key ingredient in our proof of Lemma \ref{lem:RosserScheonfeldSmoothEFConsequence} is a smoothed explicit formula for $\psi(x)$ from Rosser and Schoenfeld \cite{RosserSchoenfeld}. 

\begin{lemma}\label{lem:RosserScheonfeldSmoothEFConsequence}
Suppose that the RH is true, $H_1 := 10^7$, and
\begin{equation*}
    T_0(x) = \frac{\pi\sqrt{x}}{\log{x}} \left(1 + \frac{\log{x}}{2\pi\sqrt{x}} \right)^{-1} \left(\left(1+\frac{\log{x}}{\pi\sqrt{x}}\right)^{2} + 1\right) .
\end{equation*}
If $x \geq 10^{19}$, then
\begin{equation*}
\begin{split}
    \frac{\left|\psi(x) - x \right|}{x} 
    &\leq \frac{\log{x}}{2\pi\sqrt{x}} + \frac{\logb{2\pi} - \frac{1}{2}\logb{1-x^{-2}}}{x} \\
    &\qquad + x^{- \frac{1}{2}} \left(1 + \frac{\log{x}}{2\pi\sqrt{x}} \right) \Bigg( 
    \frac{1}{2\pi} \left(\logb{\frac{T_0(x)}{2\pi}}\right)^2 + \frac{\log{T_0(x)}}{\pi T_0(x)} \\
    &\qquad\hspace{4cm}+ \omega_1 + \frac{8\log{H_1} + 4}{H_1} - \frac{1}{2\pi} \left(\logb{\frac{H_1}{2\pi}}\right)^2
    \Bigg) .
\end{split}
\end{equation*}
\end{lemma}

\begin{proof}
Let $T_1$, $T_2$ denote non-negative real numbers, $m>0$ be an integer, $x>1$, and $0 < \delta < (1-x^{-1})/m$. Rosser and Schoenfeld proved in \cite[Lem.~8]{RosserSchoenfeld} that
\begin{equation*}
    \frac{\left|\psi(x) - x -\logb{2\pi} + \frac{1}{2}\logb{1-x^{-2}}\right|}{x} \leq \frac{m\delta}{2} +\frac{S_1(m,\delta) + S_2(m,\delta)}{\sqrt{x}} + S_3(m,\delta) + S_4(m,\delta) ,
\end{equation*}
where $R_m(u) = \left((1+u)^{m+1}+1\right)^m$, $R =  9.645908801$, 
\begin{align*}
    S_1(m,\delta) &= \sum_{\substack{\beta \leq 1/2 \\ 0 < |\gamma| <T_1}} \frac{2+m\delta}{2|\varrho|}, \\
    S_2(m,\delta) &= \sum_{\substack{\beta \leq 1/2 \\ |\gamma| > T_1}} \frac{R_m(\delta)}{\delta^m|\varrho(\varrho+1)\cdots (\varrho+m)|}, \\
    S_3(m,\delta) &= \sum_{\substack{1/2<\beta \\ 0 < |\gamma| \leq T_2}} \frac{(2+m\delta)\exp\left\{-\frac{\log{x}}{R\log(\gamma/17)}\right\}}{2|\varrho|}, \\
    S_4(m,\delta) &= \sum_{\substack{1/2<\beta\\ |\gamma| > T_2}}
    \frac{R_m(\delta)\exp\left\{-\frac{\log x}{R\log(\gamma/17)}\right\}}{\delta^m|\varrho(\varrho+1)\cdots (\varrho+m)|}.
\end{align*}
The value of $R$ corresponds to the constant in Stechkin's effective version of the classical zero-free region for $\zeta(s)$, see \cite{Stechkin}. We assume the RH, so $S_3(m,\delta) = S_4(m,\delta) = 0$ and we can re-write the preceding relationship as
\begin{equation}\label{eqn:StepOne}
    \frac{\left|\psi(x) - x \right|}{x} \leq \frac{m\delta}{2} + \frac{S_1(m,\delta) + S_2(m,\delta)}{\sqrt{x}} + \frac{\logb{2\pi} - \frac{1}{2}\logb{1-x^{-2}}}{x} ,
\end{equation}
so the constant $R$ has no effect for us. All that remains is to bound $S_1(m,\delta)$ and $S_2(m,\delta)$. 
It follows from \eqref{eqn:hafaliad3} that if {$T_1 > H_1$}, then \eqref{eqn:StepOne} can be re-written as
\begin{equation}\label{eqn:StepTwo}
\begin{split}
    \frac{\left|\psi(x) - x \right|}{x} 
    &\leq \frac{m\delta}{2} + \frac{\logb{2\pi} - \frac{1}{2}\logb{1-x^{-2}}}{x} \\
    &\qquad + x^{- \frac{1}{2}}\left(\left(1 + \frac{m\delta}{2}\right) \left(\omega_1 + \sum_{H_1 < |\gamma| \leq T_1} \frac{1}{| \gamma |}\right) + \sum_{|\gamma|>T_1} \frac{R_m(\delta)}{\delta^{m} |\gamma|^{m+1}}\right) .
\end{split}
\end{equation}
We choose 
\begin{equation}\label{eqn:choices}
    T_1 = \frac{1}{\delta} \Bigg(\frac{2 R_m(\delta)}{2 + m\delta}\Bigg)^{\frac{1}{m}}
    \quad\text{and}\quad
    \delta = \frac{\log{x}}{m\pi\sqrt{x}} ,
\end{equation}
because Rosser and Schoenfeld note that this choice of $T_1$ is close to optimum on \cite[p.~262]{RosserSchoenfeld}. Importantly, the choices in \eqref{eqn:choices} facilitate the equalities
\begin{equation*}
    \frac{R_m(\delta)}{\delta^m T_1^{m}} 
    = 1 + \frac{m\delta}{2} 
    \quad\text{and}\quad 
    T_1 = \frac{m\pi\sqrt{x}}{\log{x}} \left(1 + \frac{\log{x}}{2\pi\sqrt{x}} \right)^{-\frac{1}{m}} \left(\left(1+\frac{\log{x}}{m\pi\sqrt{x}}\right)^{m+1} + 1\right).
\end{equation*}
Note that $T_1 > 454\,161\,776 > 10^7$ on $x\geq 10^{19}$ under these choices. Therefore, it follows from \eqref{eqn:StepTwo} and the choices in \eqref{eqn:choices} that 
\begin{equation*}
\begin{split}
    \frac{\left|\psi(x) - x \right|}{x} 
    &\leq \frac{\log{x}}{2\pi\sqrt{x}} + \frac{\logb{2\pi} - \frac{1}{2}\logb{1-x^{-2}}}{x} \\
    &\qquad + x^{- \frac{1}{2}} \left(1 + \frac{\log{x}}{2\pi\sqrt{x}} \right) \left( \omega_1 + \sum_{H_1 < |\gamma| \leq T_1} \frac{1}{| \gamma |} + \sum_{|\gamma|>T_1} \frac{1}{|\gamma|^{m+1}}\right) .
\end{split}
\end{equation*}
Finally, Lemma \ref{lem:Lehman} and \eqref{eqn:Skewes} imply
\begin{equation}\label{eqn:StepFour_Alternative}
\begin{split}
    \frac{\left|\psi(x) - x \right|}{x} 
    &\leq \frac{\log{x}}{2\pi\sqrt{x}} + \frac{\logb{2\pi} - \frac{1}{2}\logb{1-x^{-2}}}{x} \\
    &\qquad + x^{- \frac{1}{2}} \left(1 + \frac{\log{x}}{2\pi\sqrt{x}} \right) \Bigg( 
    \frac{1}{2\pi} \left(\logb{\frac{T_1}{2\pi}}\right)^2 + \frac{\log{T_1}}{\pi T_1^{m}} \\
    &\qquad\hspace{4.2cm}+ \omega_1 + \frac{8\log{H_1} + 4}{H_1} - \frac{1}{2\pi} \left(\logb{\frac{H_1}{2\pi}}\right)^2
    \Bigg) .
\end{split}
\end{equation}
The result follows upon asserting $m = 1$ in \eqref{eqn:StepFour_Alternative}. 
\end{proof}

\begin{proof}[Proof of \eqref{eqn:pnt_psi}]
It follows from Theorem \ref{thm:GoldstonExplicit} that \eqref{eqn:pnt_psi} is established for all $x \geq e^{30\,369.582}$, because
\begin{equation*}
    \frac{1}{2\pi} + \frac{1.465}{\log{x}} - \frac{1.2325}{\log\log{x}} \geq \frac{1}{8\pi}
\end{equation*}
in this range. Next, we prove \eqref{eqn:pnt_psi} for $10^{19} < x< e^{30\,369.582}$. To this end, the approximation for $\psi(x)$ in Lemma \ref{lem:RosserScheonfeldSmoothEFConsequence} implies the desired outcome if and only if
\begin{align*}
    \frac{\log{x} - \log\log{x}}{4}
    &\geq 1 
    + \frac{1}{\log{x}} \left(1 + \frac{\log{x}}{2\pi\sqrt{x}} \right) \left(\logb{\frac{T_0(x)}{2\pi}}\right)^2 \\
    &\quad + \left(\logb{2\pi} - \frac{\logb{1-x^{-2}}}{2}\right) \frac{2\pi}{\sqrt{x}\log{x}} \\
    &\quad + \frac{2\pi \left(1 + \frac{\log{x}}{2\pi\sqrt{x}} \right)}{\log{x}} \Bigg( 
    \frac{\log{T_0(x)}}{\pi T_0(x)} + \omega_1 + \frac{8\log{H_1} + 4}{H_1} - \frac{1}{2\pi} \left(\logb{\frac{H_1}{2\pi}}\right)^2
    \Bigg) 
\end{align*}
in this range of $x$. Since
\begin{equation*}
\begin{split}
    \left(\frac{\logb{2\pi}}{\sqrt{x}} - \frac{\logb{1-x^{-2}}}{2\sqrt{x}}\right) \left(1 + \frac{\log{x}}{2\pi\sqrt{x}} \right)^{-1} &+ \frac{\log{T_0(x)}}{\pi T_0(x)} \\
    &\quad + \omega_1 + \frac{8\log{H_1} + 4}{H_1} < \frac{1}{2\pi} \left(\logb{\frac{H_1}{2\pi}}\right)^2 
\end{split}
\end{equation*}
for every $x \geq 10^{19}$, it suffices to prove
\begin{equation}\label{eqn:suff_cond}
    \log{x} - \log\log{x} - 4
    \geq 4 \left(1 + \frac{\log{x}}{2\pi\sqrt{x}} \right) \left(\frac{\logb{T_0(x)/2\pi}}{\log{x}}\right)^2 \log{x} 
\end{equation}
for every $10^{19} < x< e^{30\,369.582}$. Now, the coefficient of $\log{x}$ on the right-hand side of \eqref{eqn:suff_cond} increases on $x\geq 10^{19}$, so it is straightforward to compute 
\begin{equation*}
    4 \left(1 + \frac{\log{x}}{2\pi\sqrt{x}} \right) \left(\frac{\logb{T_0(x)/2\pi}}{\log{x}}\right)^2 \log{x}
    \leq \begin{cases}
        0.99865 \log{x} &\text{if $e^{9\,768.054} \leq x \leq e^{30\,369.582}$,} \\
        0.99625 \log{x} &\text{if $e^{3\,220.622} \leq x < e^{9\,768.054}$,} \\
        0.99000 \log{x} &\text{if $e^{1\,100.338} \leq x < e^{3\,220.622}$,} \\
        0.97471 \log{x} &\text{if $e^{394.532} \leq x < e^{1\,100.338}$,} \\
        0.94032 \log{x} &\text{if $e^{151.106} \leq x < e^{394.532}$,} \\
        0.87158 \log{x} &\text{if $e^{63.468} \leq x < e^{151.106}$,} \\
        0.75553 \log{x} &\text{if $10^{19} \leq x < e^{63.468}$.}
    \end{cases}
\end{equation*}
It follows from these bounds that the sufficient condition \eqref{eqn:suff_cond} is true for every $$10^{19} \leq x \leq e^{30\,369.582}.$$ Therefore, \eqref{eqn:pnt_psi} is proved for every $x \geq 10^{19}$. Finally, \eqref{eqn:Buthe} and numerical computations confirm the result in the remaining range $101 \leq x \leq 10^{19}$.
\end{proof}

\begin{proof}[Proof of \eqref{eqn:pnt_theta}]
Recall from Costa Pereira \cite[Thm.~5]{Costa} and Broadbent \textit{et al.} \cite[Cor.~5.1]{Broadbent} that for all $x\geq e^{40}$, 
\begin{equation}\label{eqn:diffs}
    0.999x^\frac{1}{2} + x^\frac{1}{3}
    < \psi(x)-\theta(x)
    < \alpha_1 x^\frac{1}{2} + \alpha_2 x^\frac{1}{3},
\end{equation}
with $\alpha_1= 1+ 1.93378 \cdot 10^{-8}$ and $\alpha_2 = 1.04320$. It follows from Theorem \ref{thm:GoldstonExplicit} and \eqref{eqn:diffs} that if the RH is true and $x\geq 10^{19}$, then
\begin{align*}
    \left|\vartheta(x) - x \right| 
    &\leq |\psi(x) - \vartheta(x)| + |\psi(x) - x| \\
    &\leq \sqrt{x}\log{x} \left(\frac{\log{x}}{8\pi} - \left(\frac{1}{2\pi} + \frac{1.465}{\log{x}}\right) \log\log{x} + 1.2325 + \frac{\alpha_1}{\log{x}} + \frac{\alpha_2}{x^\frac{1}{6}\log{x}} \right) . 
\end{align*}
With this, \eqref{eqn:pnt_theta} is established for all $x \geq e^{30\,456.256}$, because
\begin{equation*}
    \frac{1}{2\pi} + \frac{1.465}{\log{x}} - \frac{1.2325}{\log\log{x}} - \frac{\alpha_1}{\log{x}} - \frac{\alpha_2}{x^\frac{1}{6}\log{x}} \geq \frac{1}{8\pi}
\end{equation*}
in this range. Next, the approximation for $\psi(x)$ in Lemma \ref{lem:RosserScheonfeldSmoothEFConsequence} and \eqref{eqn:diffs} imply
\begin{equation*}
\begin{split}
    \frac{\left|\vartheta(x) - x \right|}{x} 
    &\leq \frac{\log{x}}{2\pi\sqrt{x}} + \frac{\logb{2\pi} - \frac{1}{2}\logb{1-x^{-2}}}{x}  + \frac{\alpha_1}{\sqrt{x}} + \frac{\alpha_2}{x^{\frac{2}{3}}} \\
    &\qquad + x^{- \frac{1}{2}} \left(1 + \frac{\log{x}}{2\pi\sqrt{x}} \right) \Bigg( 
    \frac{1}{2\pi} \left(\logb{\frac{T_0(x)}{2\pi}}\right)^2 + \frac{\log{T_0(x)}}{\pi T_0(x)} \\
    &\qquad\hspace{4.5cm}+ \omega_1 + \frac{8\log{H_1} + 4}{H_1} - \frac{1}{2\pi} \left(\logb{\frac{H_1}{2\pi}}\right)^2
    \Bigg) .
\end{split}
\end{equation*}
Arguing similar lines to earlier, this implies the result \eqref{eqn:pnt_theta} is true whenever the sufficient condition \eqref{eqn:suff_cond} is true. We know \eqref{eqn:suff_cond} is true for all $10^{19} \leq x \leq e^{30\,369.582}$ by our earlier analysis, and if $e^{30\,369.582} \leq x \leq e^{30\,456.276}$, then 
\begin{align*}
    4 \left(1 + \frac{\log{x}}{2\pi\sqrt{x}} \right) \left(\frac{\logb{T_0(x)/2\pi}}{\log{x}}\right)^2 \log{x}
    &\leq 0.99865 \log{x} .
\end{align*}
Now, $0.99865 \log{x} \leq \log{x} - \log\log{x} - 4$ for every $e^{30\,369.582} \leq x \leq e^{30\,456.276}$, so \eqref{eqn:pnt_theta} is proved for every $x \geq 10^{19}$. 
Finally, B\"{u}the tells us in \cite[Thm.~2]{Buthe} that if $1\,423 \leq x \leq 10^{19}$, then
\begin{equation}\label{eqn:Buthe2}
    |\vartheta(x) - x| \leq 1.95\sqrt{x} .
\end{equation}
Therefore, \eqref{eqn:Buthe2} and numerical computations confirm the result in the remaining range $2\,657 \leq x \leq 10^{19}$. 
\end{proof}

\subsection{Proof of Corollary \ref{cor:broader}}\label{ssec:broader}

It follows from partial summation that
\begin{align}
    \sum_{p\leq x} \frac{\log{p}}{p}
    &= \log{x} + K_1 + \frac{\vartheta(x) - x}{x} - \int_x^\infty \frac{\vartheta(t) - t}{t^2}\,dt 
    \quad\text{and}\label{eqn:brr_1}\\
    \sum_{p\leq x} \frac{1}{p}
    &= \log\log{x} + K_2 + \frac{\vartheta(x) - x}{x\log{x}} - \int_x^\infty \frac{\log{t} + 1}{t^2(\log{t})^2} (\vartheta(t) - t)\,dt , \label{eqn:brr_2}
\end{align}
in which
\begin{align*}
    K_1 &= 1 - \log{2} + \int_2^\infty \frac{\vartheta(t) - t}{t^2}\,dt
    \quad\text{and}\\
    K_2 &=  \frac{1}{\log{2}} - \log\log{2} + \int_2^\infty \frac{\log{t} + 1}{t^2(\log{t})^2} (\vartheta(t) - t)\,dt .
\end{align*}
Now, $K_1 = E$ and $K_2 = B$, because $K_1$ and $K_2$ are constant and inexplicit versions of Mertens' theorems tell us that
\begin{align*}
    \sum_{p\leq x} \frac{\log{p}}{p}
    = \log{x} + E + o(1) 
    \quad\text{and}\quad
    \sum_{p\leq x} \frac{1}{p}
    = \log\log{x} + B + o(1) 
    \quad\text{as}\quad x\to\infty .
\end{align*}
Using these observations in conjunction with \eqref{eqn:pnt_theta}, we prove \eqref{eqn:moi_1} and \eqref{eqn:moi_2} as follows.

\begin{proof}[Proof of \eqref{eqn:moi_1}]
Apply \eqref{eqn:pnt_theta} in \eqref{eqn:brr_1} to see that if the RH is true and $x \geq 10^{19}$, then
\begin{align*}
    \left| \sum_{p\leq x} \frac{\log{p}}{p} - \log{x} - E \right| 
    &\leq \frac{\log{x} (\log{x} - \log\log{x})}{8\pi\sqrt{x}} + \int_x^\infty \frac{\log{t}(\log{t} - \log\log{t})}{8\pi t^{3/2}}\,dt \\
    &\leq \frac{\log{x} (\log{x} - 3.77847)}{8\pi\sqrt{x}} + \int_x^\infty \frac{\log{t}(\log{t} - 3.77847)}{8\pi t^{3/2}}\,dt \\
    &= \frac{3(\log{x})^2 - 3.33541 \log{x} + 0.88612}{8\pi \sqrt{x}} 
    < \frac{3(\log{x})^2}{8\pi \sqrt{x}} .
\end{align*}
Insert \eqref{eqn:pnt_theta} and \eqref{eqn:Buthe2} into \eqref{eqn:brr_1} to see that if $10^6 \leq x \leq 10^{19}$, then
\begin{align*}
    \left| \sum_{p\leq x} \frac{\log{p}}{p} - \log{x} - E \right| 
    &< 1.95 \left(\frac{1}{\sqrt{x}} + \int_x^{10^{19}} \frac{dt}{t^{3/2}}\right) + \frac{1}{8\pi} \int_{10^{19}}^\infty \frac{(\log{t})^2}{t^{3/2}}\,dt \\
    &\leq 1.95 \left(\frac{3}{\sqrt{x}} - \frac{2}{10^{19/2}}\right) + \frac{10^{19/2}(\log{x})^2}{8\pi(19\log{10})^2 \sqrt{x} } \int_{10^{19}}^\infty \frac{(\log{t})^2}{t^{3/2}}\,dt \\
    &< \frac{3\cdot 1.95\cdot 8\pi + 2.2 (\log{x})^2}{8\pi\sqrt{x}} 
    < \frac{3(\log{x})^2}{8\pi\sqrt{x}} .
\end{align*}
To complete our proof, we have verified \eqref{eqn:moi_1} for $x < 10^6$ using computer checks.
\end{proof}

\begin{proof}[Proof of \eqref{eqn:moi_2}]
Apply \eqref{eqn:pnt_theta} in \eqref{eqn:brr_2} to see that if the RH is true and $x \geq 10^{19}$, then
\begin{align*}
    \left|\sum_{p\leq x} \frac{1}{p} - \log\log{x} - B\right| 
    &\leq \frac{\log{x} - 3.77847}{8\pi\sqrt{x}} + \int_x^\infty \frac{(\log{t} - 3.77847)(\log{t} + 1)}{8\pi t^{3/2}\log{t}}\,dt \\
    &\leq \frac{\log{x} - 3.77847}{8\pi\sqrt{x}} + \left( 1 + \frac{1}{\log{x}} \right) \int_x^\infty \frac{\log{t} - 3.77847}{8\pi t^{3/2}}\,dt \\
    &= \frac{\log{x} - 3.77847}{8\pi\sqrt{x}} 
    + \left( 1 + \frac{1}{\log{x}} \right) \left(\frac{2(\log{x} + 1) - 2\cdot 3.77847}{8\pi\sqrt{x}}\right)
    \\
    &= \frac{3(\log{x} - 3.77847) + 2\left(2 - \frac{2.77847}{\log{x}}\right)}{8\pi\sqrt{x}} < \frac{3\log{x}}{8\pi\sqrt{x}}.
\end{align*}
Insert \eqref{eqn:pnt_theta} and \eqref{eqn:Buthe2} into \eqref{eqn:brr_2} to see that if $10^6 \leq x \leq 10^{19}$, then
\begin{align*}
    \left|\sum_{p\leq x} \frac{1}{p} - \log\log{x} - B\right| 
    &\leq 1.95\left(\frac{1}{\sqrt{x}\log{x}} + \int_x^{10^{19}} \frac{\log{t} + 1}{t^{3/2}(\log{t})^2}\,dt\right) + \frac{1}{8\pi} \int_{10^{19}}^\infty \frac{\log{t} + 1}{t^{3/2}}\,dt \\
    &< 1.95\left(\frac{1}{\sqrt{x}\log{x}} + \frac{1.08}{\log{x}} \int_x^{\infty} \frac{dt}{t^{3/2}}\,dt\right) + \frac{2.14\log{x}}{8\pi\sqrt{x}} \\
    &= \left(\frac{3.16\cdot 1.95\cdot 8\pi}{(\log{x})^2} + 2.14\right) \frac{\log{x}}{8\pi\sqrt{x}} 
    \leq \frac{2.95139\log{x}}{8\pi\sqrt{x}} .
\end{align*}
To complete our proof, we have verified \eqref{eqn:moi_2} for $x < 10^6$ using computer checks.
\end{proof}

Next, Ingham tells us in \cite{Ingham} that we can write
\begin{equation*}
    B = C + \sum_{p}\left(\logb{1-\frac{1}{p}} + \frac{1}{p}\right) .
\end{equation*}
Recall that we have already proved that if the RH is true and $x\geq 10^6$, then
\begin{equation*}
    \left|\sum_{p\leq x} \frac{1}{p} - \log\log{x} - B\right| 
    \leq \mathcal{R}(x) , \quad\text{where}\quad
    \mathcal{R}(x) = 
    \begin{cases}
        \frac{2.95139\log{x}}{8\pi\sqrt{x}} &\text{if $10^6 \leq x < 10^{19}$,}\\
        \frac{3(\log{x} - 3.37784) + 4}{8\pi\sqrt{x}} &\text{if $x \geq 10^{19}$.}
    \end{cases}
\end{equation*}
It follows from this observation that if the RH is true and $x\geq 10^6$, then
\begin{align*}
    B 
    &= C + \sum_{p\leq x}\left(\logb{1-\frac{1}{p}} + \frac{1}{p}\right) + \Theta(x) \\
    &= C + \log\prod_{p\leq x} \left(1-\frac{1}{p}\right) + \log\log{x} + B + \Theta(x) + O^*(\mathcal{R}(x)) ,
\end{align*}
where $O^*$ is big-$O$ notation with implied constant one and
\begin{equation*}
    \Theta(x) = \sum_{p > x} \left(\logb{1-\frac{1}{p}}+\frac{1}{p}\right) .
\end{equation*}
Therefore, we have shown
\begin{align*}
    \left|\log\prod_{p\leq x} \left(1-\frac{1}{p}\right) 
    + C + \log\log{x} + \Theta(x) \right| &\leq \mathcal{R}(x) ,
\end{align*}
and hence we have
\begin{equation*}
    e^{C} \log{x} \prod_{p\leq x} \left(1-\frac{1}{p}\right) 
    = e^{- \Theta(x) + O^*(\mathcal{R}(x))} .
\end{equation*}
Now, Rosser and Schoenfeld proved in \cite[(8.2)]{RosserSchoenfeld1962} that 
\begin{equation*}
    0 > \Theta(x) > - \frac{1.02}{(x-1)\log{x}} = - \Theta_0(x), 
    \quad\text{for}\quad
    x > 1 ,
\end{equation*}
so we can see
\begin{align*}
    e^{C} \log{x} \prod_{p\leq x} \left(1-\frac{1}{p}\right) 
    &> e^{- \mathcal{R}(x)} > 1 - \mathcal{R}(x) 
    \quad\text{and}\\
    e^{C} \log{x} \prod_{p\leq x} \left(1-\frac{1}{p}\right) 
    &< e^{\Theta_0(x) + \mathcal{R}(x)} 
    \leq 1 + \Theta_0(x) + \mathcal{R}(x) + 0.501(\Theta_0(x) + \mathcal{R}(x))^2 .
\end{align*}
Equivalently,
\begin{equation*}
    \left|e^C \log{x} \prod_{p\leq x} \left(1-\frac{1}{p}\right) - 1\right|
    < \Theta_0(x) + \mathcal{R}(x) + 0.501 (\Theta_0(x) + \mathcal{R}(x))^2
\end{equation*}
Similarly, we have
\begin{equation*}
    \frac{e^{-C}}{\log{x}} \prod_{p\leq x} \left(1-\frac{1}{p}\right)^{-1} 
    = e^{\Theta(x) + O^*(\mathcal{R}(x))} ,
\end{equation*}
so we can also see
\begin{equation*}
    \left|\frac{e^{-C}}{\log{x}} \prod_{p\leq x} \left(1-\frac{1}{p}\right)^{-1} - 1\right|
    < \Theta_0(x) + \mathcal{R}(x) + 0.501 (\Theta_0(x) + \mathcal{R}(x))^2 .
\end{equation*}
The result follows from observing that if $10^6 \leq x < 10^{19}$, then 
\begin{align*}
    &\Theta_0(x) + \mathcal{R}(x) + 0.501(\Theta_0(x) + \mathcal{R}(x))^2 \\
    &\qquad\qquad\leq \frac{1.02}{(x-1)\log{x}} + \frac{2.95139\log{x}}{8\pi\sqrt{x}} + 0.501\left(\frac{1.02}{(x-1)\log{x}} + \frac{2.95139\log{x}}{8\pi\sqrt{x}}\right)^2 \\
    &\qquad\qquad< \frac{3\log{x}}{8\pi\sqrt{x}},
\end{align*}
and if $x \geq 10^{19}$, then
\begin{align*}
    &\Theta_0(x) + \mathcal{R}(x) + 0.501(\Theta_0(x) + \mathcal{R}(x))^2 \\
    &\quad\leq \frac{1.02}{(x-1)\log{x}} + \frac{3(\log{x} - 3.37784) + 4}{8\pi\sqrt{x}} \\
    &\qquad\hspace{4cm} + 0.501\left(\frac{1.02}{(x-1)\log{x}} + \frac{3(\log{x} - 3.37784) + 4}{8\pi\sqrt{x}}\right)^2 \\
    &\quad< \frac{3\log{x}}{8\pi\sqrt{x}} .
\end{align*}
It follows that \eqref{eqn:moi_3} and \eqref{eqn:moi_4} are proved for any $x \geq 10^{19}$. To complete the proof of \eqref{eqn:moi_3} and \eqref{eqn:moi_4}, we have verified each result for $x < 10^6$ using computer checks.

\begin{remark}
The constant $0.501$ can be reduced to $0.5 + \varepsilon$ for a very small $\varepsilon > 0$ at large $x$ using a simple Taylor series argument. 
\end{remark}

\bibliographystyle{amsplain}
\bibliography{refs}

\end{document}